\documentclass[12pt,reqno]{amsart}
\usepackage{amsmath,amssymb,amsfonts,amscd,latexsym,amsthm,mathrsfs}
\usepackage[usenames]{color}
\usepackage[unicode]{hyperref}
\makeatletter
\renewcommand\thesubsection{\@arabic\c@subsection}
\makeatother

\renewcommand\pmod[1]{\;(\operatorname{mod}#1)}
\newtheorem{theorem}{Theorem}

\theoremstyle{remark}

\begin{document}

\title{A non-ordinary (prime) note}

\date{31 August 2024}

\author{Wadim Zudilin}
\address{Department of Mathematics, IMAPP, Radboud University, PO Box 9010, 6500~GL Nijmegen, Netherlands}
\email{w.zudilin@math.ru.nl}

\subjclass[2020]{Primary 11F33; Secondary 11F30, 11P83, 33C20}

\begin{abstract}
Given a newform with the Fourier expansion $\sum_{n=1}^\infty b(n)q^n\in\mathbb Z[[q]]$, a prime $p$ is said to be non-ordinary if $p\mid b(p)$.
We exemplify several newforms of weight 4 for which the latter divisibility implies a stronger divisibility\,---\,a property that may be thought unlikely to happen too often.
\end{abstract}

\maketitle

For a normalised cusp eigenform (\emph{aka} newform) $f(\tau)=\sum_{n=1}^\infty b(n)q^n\in\mathbb Z[[q]]$ with $q=e^{2\pi i\tau}$, consider the question
of nonvanishing $b(p)$ modulo~$p$. The primes for which such nonvanishing takes place are known as \emph{ordinary primes} (for the form $f(\tau)$); ones for which $b(p)\equiv0\pmod p$ are non-ordinary. It is widely accepted (see \cite{Go97} for the level~1 case) that the (Dirichlet) density of ordinary primes is 1 for non-CM newforms, though already the problem of showing that there are infinitely many of them remains open for any concrete such newform of weight greater than~$3$.
Slightly more can be said in the case when a newform $f(\tau)$ is CM\,---\,see~\cite{LM14}.

In this note we focus on weight~$4$ and very particular choices of eigenforms but we do not pretend to demonstrate the infinitude of ordinary primes.
We rather explain that the non-ordinary primes imply a significantly stronger divisibility property than just $p\mid b(p)$, thus giving a heuristical argument why they are unlikely to show up `too often'.

In what follows $(a)_k=\Gamma(a+k)/\Gamma(a)=\prod_{j=0}^{k-1}(a+j)$ denotes the Pochhammer symbol.

\begin{theorem}
\label{th1}
A prime $p>2$ is non-ordinary for the newform
\begin{equation}
\eta(2\tau)^4\eta(4\tau)^4=q\prod_{m=1}^\infty(1-q^{2m})^4(1-q^{4m})^4
=\sum_{n=1}^\infty b(n)q^n
\label{e01}
\end{equation}
if and only if the degree $4(p-1)$ polynomial
\[
Q_p(a)=2^{4(p-1)}(a+1)_{p-1}^4\cdot\sum_{k=0}^{p-1}\frac{(a+\frac12)_k^4}{(a+1)_k^4}\in\mathbb Z[a]
\]
has \emph{all} its coefficients divisible by~$p$.
\end{theorem}

The newform \eqref{e01} happens to be the (unique) cusp eigenform of weight~4 on $\Gamma_0(8)$ and it has a certain historical significance.
Motivated by congruences arising from `formal group laws' \cite{St87,SB85}, F.~Beukers proved \cite{Be87} in 1987 a result for the Ap\'ery numbers, which can be equivalently stated as
\begin{equation}
\sum_{k=0}^{p-1}\frac{(\frac12)_k^4}{k!^4}\equiv b(p)\pmod{p}
\label{e02}
\end{equation}
for primes $p>2$.
Here and below the congruence $A\equiv B\pmod{p^\ell}$ for two \emph{rational} numbers is understood as $A-B\in p^\ell\mathbb Z_p$. When $b(p)\equiv0\pmod p$, congruence \eqref{e02} translates into $Q_p(0)\equiv0\pmod p$, thus demonstrating that the constant term of $Q_p(a)$ is divisible by~$p$.
Ten years later L.~Van Hamme \cite[Conjecture~(M.2)]{vH97}  observed numerically that
the congruence \eqref{e02} is valid modulo $p^3$. This conjecture was finally settled by T.~Kilbourn in~\cite{Ki06} built on an earlier work of S.~Ahlgren and K.~Ono in \cite{AO00} on the modularity of the Calabi--Yau threefold $\sum_{j=1}^4(x_j+x_j^{-1})=0$.
Furthermore, the ($p$-adic) congruence \eqref{e02} and its extensions possess an Archimedean counterpart
\begin{equation}
\sum_{k=0}^\infty\frac{(\frac12)_k^4}{k!^4}
=\frac{16L(f,2)}{\pi^2},
\label{e03}
\end{equation}
where $L(f,s)$ denotes the $L$-function of the modular form \eqref{e01}.
Equality \eqref{e03} was established independently in~\cite{RWZ15} and \cite[Section~7]{Za16}.

For the first two ordinary odd primes $p=11$ and $3137$ (the only ones up to $20\,000$), one can easily verify the divisibility offered in Theorem~\ref{th1}.
In fact, $b(11)=-44$ and $b(3137)=66\cdot3137$ are nonzero, and one may further suspect that $b(n)$~is never zero for odd $n$ viewing this as a baby version of Lehmer's question from \cite{Le43} about the nonvanishing of the Fourier coefficients of the modular invariant $\eta(\tau)^{24}=q\prod_{m=0}^\infty(1-q^m)^{24}$.

\begin{proof}[Proof of Theorem~\textup{\ref{th1}}]
For the proof%
\footnote{The role of $q$ within the proof is different from that outside it; this should not cause any confusion though.}
we recall the notation $(a;q)_k=\prod_{j=0}^{k-1}(1-aq^j)$ of the $q$-Pochhammer symbol and the identity, in fact only its special case $\ell_1=\ell_2=(n+1)/2$ with $n$~odd, proven in~\cite{KZ24}:
\begin{align}
F_n(a;\zeta)
&=\frac{n^2a^{n-1}}{(1+a+a^2+\dots+a^{n-1})^2}\,G_n(a;\zeta)F_n(1;\zeta)
\nonumber\\
&=\frac{a^{n-1}\prod_{j=1}^{n-1}(\zeta^j-1)^2}{\prod_{j=1}^{n-1}(a-\zeta^j)^2}\,G_n(a;\zeta)F_n(1;\zeta),
\label{e04}
\end{align}
where
\[
F_n(a;q)=\sum_{k=0}^{n-1}\frac{(aq^{(n+1)/2};q)_k^2(aq^{(-n+1)/2};q)_k^2}{(aq;q)_k^4}\,q^k,
\quad
G_n(a;q)=\prod_{j=1}^{(n-1)/2}\frac{(a-q^j)^2}{(1-aq^j)^2},
\]
and $\zeta=\zeta_n$ is any primitive $n$th root of unity.
The equality in \eqref{e04} translates into the congruence
\begin{equation}
F_n(a;q)\equiv\frac{a^{n-1}\prod_{j=1}^{n-1}(q^j-1)^2}{\prod_{j=1}^{(n-1)/2}(1-aq^j)^2\cdot\prod_{j=(n+1)/2}^{n-1}(a-q^j)^2}\,F_n(1;q)
\label{e05}
\end{equation}
modulo the cyclotomic polynomial $\Phi_n(q)$. When $n=p$ is prime, we have $\Phi_p(1)=p$;
therefore, substituting $q^a$ for~$a$ and passing to the limit as $q\to1$ in \eqref{e05} lead us to the congruence
\[
\sum_{k=0}^{p-1}\frac{(a+\frac{p+1}2)_k^2(a+\frac{-p+1}2)_k^2}{(a+1)_k^4}
\equiv\frac{(p-1)!^2}{(a+1)_{(p-1)/2}^4}
\cdot\sum_{k=0}^{p-1}\frac{(\frac{p+1}2)_k^2(\frac{-p+1}2)_k^2}{k!^4}\pmod p,
\]
hence
\begin{equation}
\sum_{k=0}^{p-1}\frac{(a+\frac12)_k^4}{(a+1)_k^4}
\equiv\frac1{(a+1)_{(p-1)/2}^4}\cdot\sum_{k=0}^{p-1}\frac{(\frac12)_k^4}{k!^4}
\equiv\frac{b(p)}{(a+1)_{(p-1)/2}^4}\pmod p
\label{e06}
\end{equation}
in view of \eqref{e02}.
It remains to clean up the denominator on the left-hand side in~\eqref{e05}.
\end{proof}

The above argument actually shows that modulo~$p$,
\[
Q_p(a)\equiv 2^{4(p-1)}(-a+1)_{(p-1)/2}^4\cdot b(p)
\equiv(-a+1)_{(p-1)/2}^4\cdot b(p)
\]
coefficient-wise, perhaps making the final stronger divisibility in Theorem~\ref{th1} less surprising.

\medskip
A uniform treatment of thirteen more cases
\[
\sum_{k=0}^{p-1}\frac{(s_1)_k(s_2)_k(1-s_1)_k(1-s_2)_k}{k!^4}\equiv b_{s_1,s_2}(p)\pmod{p^3}
\]
is given recently in \cite{LTYZ21} (see also \cite{Zu20} for a $q$-alternative of the arithmetic part\,---\,this paper motivated the discovery of the principal result in \cite{KZ24}), with the explicit identification of the weight~$4$ newforms $f_{s_1,s_2}(\tau)=\sum_{n=1}^\infty b_{s_1,s_2}(n)q^n$.
Different specialisations of \cite[Theorem~1]{KZ24} imply that
\[
(a+1)_{p-1}^4\cdot\sum_{k=0}^{p-1}\frac{(a+s_1)_k(a+s_2)_k(a+1-s_1)_k(a+1-s_2)_k}{(a+1)_k^4}\in p\mathbb Z_p[a]
\]
whenever a prime $p>5$ is non-ordinary for the corresponding modular form $f_{s_1,s_2}(\tau)$.
Though it indeed looks quite unlikely to get this strong divisibility for an infinite range of~$p$, one particular (CM!) example $f_{1/4,1/3}(\tau)=\eta(3\tau)^8=q\prod_{m=1}^\infty(1-q^{3m})^8$ clearly displays that $b_{1/4,1/3}(p)=0$ for primes $p\equiv2\pmod3$, so that the coefficients of the polynomials
\[
2^{4(p-1)}\,3^{2(p-1)}\,(a+1)_{p-1}^4\cdot\sum_{k=0}^{p-1}\frac{(a+\frac14)_k(a+\frac13)_k(a+\frac23)_k(a+\frac34)_k}{(a+1)_k^4}\in\mathbb Z[a]
\]
are always divisible by $p$ for such odd primes~$p$.
It may be interesting to have a different proof of the fact, to gain a better understanding of the CM phenomenon for newforms.


\end{document}